\documentclass[11pt]{amsart}
\usepackage{}
\usepackage{amssymb}
\usepackage{amsfonts}
\usepackage{mathrsfs}
\usepackage{latexsym}
\usepackage{graphicx}
\usepackage{amscd,amssymb,amsmath,amsbsy,amsthm,amsfonts}
\usepackage[all]{xy}
\usepackage[colorlinks,plainpages,backref,urlcolor=blue]{hyperref}
\usepackage{verbatim}
\usepackage{enumerate}

\usepackage [english]{babel}
\usepackage [autostyle, english = american]{csquotes}
\MakeOuterQuote{"}

\newcommand {\Omit}[1]{}

\usepackage{tikz-cd}

\usepackage{tikz}
\usetikzlibrary{automata,positioning}

\usetikzlibrary{shapes,shadows}

\usetikzlibrary{arrows,calc}
\tikzset{
>=stealth',
help lines/.style={dashed, thick},
axis/.style={<->},
important line/.style={thick},
connection/.style={thick, dotted},
}
\usetikzlibrary{patterns}
\newlength{\hatchspread}
\newlength{\hatchthickness}
\tikzset{hatchspread/.code={\setlength{\hatchspread}{#1}},
         hatchthickness/.code={\setlength{\hatchthickness}{#1}}}
\tikzset{hatchspread=3pt,
         hatchthickness=0.4pt}
\pgfdeclarepatternformonly[\hatchspread,\hatchthickness]% variables
   {custom north west lines}% name
   {\pgfqpoint{-2\hatchthickness}{-2\hatchthickness}}% lower left corner
   {\pgfqpoint{\dimexpr\hatchspread+2\hatchthickness}{\dimexpr\hatchspread+2\hatchthickness}}% upper right corner
   {\pgfqpoint{\hatchspread}{\hatchspread}}% tile size
   {% shape description
    \pgfsetlinewidth{\hatchthickness}
    \pgfpathmoveto{\pgfqpoint{0pt}{\hatchspread}}
    \pgfpathlineto{\pgfqpoint{\dimexpr\hatchspread+0.15pt}{-0.15pt}}
    \pgfusepath{stroke}
   }
\newcommand{\nc}{\newcommand}
\nc{\rnc}{\renewcommand}
\nc{\bb}[1]{{\mathbb #1}}
\nc{\bbA}{\bb{A}}\nc{\bbB}{\bb{B}}\nc{\bbC}{\bb{C}}\nc{\bbD}{\bb{D}}
\nc{\bbE}{\bb{E}}\nc{\bbF}{\bb{F}}\nc{\bbG}{\bb{G}}\nc{\bbH}{\bb{H}}
\nc{\bbI}{\bb{I}}\nc{\bbJ}{\bb{J}}\nc{\bbK}{\bb{K}}\nc{\bbL}{\bb{L}}
\nc{\bbM}{\bb{M}}\nc{\bbN}{\bb{N}}\nc{\bbO}{\bb{O}}\nc{\bbP}{\bb{P}}
\nc{\bbQ}{\bb{Q}}\nc{\bbR}{\bb{R}}\nc{\bbS}{\bb{S}}\nc{\bbT}{\bb{T}}
\nc{\bbU}{\bb{U}}\nc{\bbV}{\bb{V}}\nc{\bbW}{\bb{W}}\nc{\bbX}{\bb{X}}
\nc{\bbY}{\bb{Y}}\nc{\bbZ}{\bb{Z}}
\nc{\mbf}[1]{{\mathbf #1}}
\nc{\bfA}{\mbf{A}}\nc{\bfB}{\mbf{B}}\nc{\bfC}{\mbf{C}}\nc{\bfD}{\mbf{D}}
\nc{\bfE}{\mbf{E}}\nc{\bfF}{\mbf{F}}\nc{\bfG}{\mbf{G}}\nc{\bfH}{\mbf{H}}
\nc{\bfI}{\mbf{I}}\nc{\bfJ}{\mbf{J}}\nc{\bfK}{\mbf{K}}\nc{\bfL}{\mbf{L}}
\nc{\bfM}{\mbf{M}}\nc{\bfN}{\mbf{N}}\nc{\bfO}{\mbf{O}}\nc{\bfP}{\mbf{P}}
\nc{\bfQ}{\mbf{Q}}\nc{\bfR}{\mbf{R}}\nc{\bfS}{\mbf{S}}\nc{\bfT}{\mbf{T}}
\nc{\bfU}{\mbf{U}}\nc{\bfV}{\mbf{V}}\nc{\bfW}{\mbf{W}}\nc{\bfX}{\mbf{X}}
\nc{\bfY}{\mbf{Y}}\nc{\bfZ}{\mbf{Z}}
\nc{\bfa}{\mbf{a}}\nc{\bfb}{\mbf{b}}\nc{\bfc}{\mbf{c}}\nc{\bfd}{\mbf{d}}
\nc{\bfe}{\mbf{e}}\nc{\bff}{\mbf{f}}\nc{\bfg}{\mbf{g}}\nc{\bfh}{\mbf{h}}
\nc{\bfi}{\mbf{i}}\nc{\bfj}{\mbf{j}}\nc{\bfk}{\mbf{k}}\nc{\bfl}{\mbf{l}}
\nc{\bfm}{\mbf{m}}\nc{\bfn}{\mbf{n}}\nc{\bfo}{\mbf{o}}\nc{\bfp}{\mbf{p}}
\nc{\bfq}{\mbf{q}}\nc{\bfr}{\mbf{r}}\nc{\bfs}{\mbf{s}}\nc{\bft}{\mbf{t}}
\nc{\bfu}{\mbf{u}}\nc{\bfv}{\mbf{v}}\nc{\bfw}{\mbf{w}}\nc{\bfx}{\mbf{x}}
\nc{\bfy}{\mbf{y}}\nc{\bfz}{\mbf{z}}

\newcommand{\op}{\text{op}}

\nc{\mcal}[1]{{\mathcal #1}}
\nc{\calA}{\mcal{A}}\nc{\calB}{\mcal{B}}\nc{\calC}{\mcal{C}}\nc{\calD}{\mcal{D}}
\nc{\calE}{\mcal{E}} \nc{\calF}{\mcal{F}}\nc{\calG}{\mcal{G}}\nc{\calH}{\mcal{H}}
\nc{\calI}{\mcal{I}}\nc{\calJ}{\mcal{J}}\nc{\calK}{\mcal{K}}\nc{\calL}{\mcal{L}}
\nc{\calM}{\mcal{M}}\nc{\calN}{\mcal{N}}\nc{\calO}{\mcal{O}}\nc{\calP}{\mcal{P}}
\nc{\calQ}{\mcal{Q}}\nc{\calR}{\mcal{R}}\nc{\calS}{\mcal{S}}\nc{\calT}{\mcal{T}}
\nc{\calU}{\mcal{U}}\nc{\calV}{\mcal{V}}\nc{\calW}{\mcal{W}}\nc{\calX}{\mcal{X}}
\nc{\calY}{\mcal{Y}}\nc{\calZ}{\mcal{Z}}
\nc{\fA}{\frak{A}}\nc{\fB}{\frak{B}}\nc{\fC}{\frak{C}} \nc{\fD}{\frak{D}}
\nc{\fE}{\frak{E}}\nc{\fF}{\frak{F}}\nc{\fG}{\frak{G}}\nc{\fH}{\frak{H}}
\nc{\fI}{\frak{I}}\nc{\fJ}{\frak{J}}\nc{\fK}{\frak{K}}\nc{\fL}{\frak{L}}
\nc{\fM}{\frak{M}}\nc{\fN}{\frak{N}}\nc{\fO}{\frak{O}}\nc{\fP}{\frak{P}}
\nc{\fQ}{\frak{Q}}\nc{\fR}{\frak{R}}\nc{\fS}{\frak{S}}\nc{\fT}{\frak{T}}
\nc{\fU}{\frak{U}}\nc{\fV}{\frak{V}}\nc{\fW}{\frak{W}}\nc{\fX}{\frak{X}}
\nc{\fY}{\frak{Y}}\nc{\fZ}{\frak{Z}}
\nc{\fa}{\frak{a}}\nc{\fb}{\frak{b}}\nc{\fc}{\frak{c}} \nc{\fd}{\frak{d}}
\nc{\fe}{\frak{e}}\nc{\fFf}{\frak{f}}\nc{\fg}{\frak{g}}\nc{\fh}{\frak{h}}
\nc{\fri}{\frak{i}}\nc{\fj}{\frak{j}}\nc{\fk}{\frak{k}}\nc{\fl}{\frak{l}}
\nc{\fm}{\frak{m}}\nc{\fn}{\frak{n}}\nc{\fo}{\frak{o}}\nc{\fp}{\frak{p}}
\nc{\fq}{\frak{q}}\nc{\fr}{\frak{r}}\nc{\fs}{\frak{s}}\nc{\ft}{\frak{t}}
\nc{\fu}{\frak{u}}\nc{\fv}{\frak{v}}\nc{\fw}{\frak{w}}\nc{\fx}{\frak{x}}
\nc{\fy}{\frak{y}}\nc{\fz}{\frak{z}}

\newtheorem{theorem}{Theorem}[section]

\newtheorem{prop}[theorem]{Proposition}

\theoremstyle{definition}

\newtheorem{example}[theorem]{Example}
\newtheorem{remark}[theorem]{Remark}

\newtheorem{thm}{Theorem}

\DeclareMathOperator{\Frame}{fr}

 \DeclareMathOperator{\GL}{GL}
  \DeclareMathOperator{\Th}{Th}
    \DeclareMathOperator{\aff}{aff}
  \DeclareMathOperator{\cRep}{\mathfrak{R}ep}

\DeclareMathOperator{\Hilb}{{Hilb}}

   \DeclareMathOperator{\sph}{sph}

\DeclareMathOperator{\SH}{SH}
\DeclareMathOperator{\Ell}{\mathcal{E}\textit{ll}}

\DeclareMathOperator{\Gr}{Gr}

\DeclareMathOperator{\CH}{CH}

\DeclareMathOperator{\MO}{MO}

\DeclareMathOperator{\inc}{in}
\DeclareMathOperator{\out}{out}
\DeclareMathOperator{\rat}{rat}

\newcommand{\g}{\mathfrak{g}}

\DeclareMathOperator{\Rep}{Rep}

\DeclareMathOperator{\Sh}{Sh}

\newcommand{\inj}{\hookrightarrow}

\newcommand{\pt}{\text{pt}}

\newcommand{\Z}{\bbZ}
\newcommand{\C}{\bbC}

\newcommand{\N}{\bbN}

\DeclareMathOperator{\fac}{fac}

\newcommand{\loc}{loc}
\DeclareMathOperator{\BM}{BM}

\newcount\cols
{\catcode`,=\active\catcode`|=\active
 \gdef\Young(#1){\hbox{$\vcenter
 {\mathcode`,="8000\mathcode`|="8000
  \def,{\global\advance\cols by 1 &}%
  \def|{\cr
        \multispan{\the\cols}\hrulefill\cr
        &\global\cols=2 }%
  \offinterlineskip\everycr{}\tabskip=0pt
  \dimen0=\ht\strutbox \advance\dimen0 by \dp\strutbox
  \halign
   {\vrule height \ht\strutbox depth \dp\strutbox##
    &&\hbox to \dimen0{\hss$##$\hss}\vrule\cr
    \noalign{\hrule}&\global\cols=2 #1\crcr
    \multispan{\the\cols}\hrulefill\cr%
   }
 }$}}
}

\begin{document}
\title{How to sheafify an elliptic quantum group}
%\date{\today}

\author[Y.~Yang]{Yaping~Yang}
\address{School of Mathematics and Statistics, The University of Melbourne, 813 Swanston Street, Parkville VIC 3010, Australia}
\email{yaping.yang1@unimelb.edu.au}

\author[G.~Zhao]{Gufang~Zhao}
\address{Institute of Science and Technology Austria,
Am Campus, 1,
Klosterneuburg 3400,
Austria}
\email{gufang.zhao@ist.ac.at}

\begin{abstract}
These lecture notes are based on Yang's talk at the MATRIX program \textit{Geometric R-Matrices: from Geometry to Probability}, at the University of Melbourne, Dec.18-22, 2017, and Zhao's talk at Perimeter Institute for Theoretical Physics in January 2018. 
We give an introductory survey of the results in  \cite{YZell}. We discuss a sheafified elliptic quantum group associated to any symmetric Kac-Moody Lie algebra. 
The sheafification is obtained by applying the equivariant elliptic cohomological theory to the moduli space of representations of a preprojective algebra. 
By construction, the elliptic quantum group naturally acts on the equivariant elliptic cohomology of Nakajima quiver varieties. 
As an application, we obtain a relation between the sheafified elliptic quantum group and the global affine Grassmannian over an elliptic curve. 
\end{abstract}
\maketitle
%\tableofcontents

\section{Motivation}
The parallelism of the following three different kinds of mathematical objects was observed in \cite{GKV} by Ginzburg-Kapranov-Vasserot.

\tikzstyle{elli}=[draw, ellipse, minimum height=3em, text width=11em, text centered, node distance=10em]
\[
\begin{tikzpicture}
\node [elli] (0, 0) {(iii): Oriented Cohomology Theory: $\text{CH}, K, \Ell$};
\node [elli, above left ,  xshift=-4.5em, yshift=4em] (-1, 1) 
{(i): 1-dimensional algebraic groups $\mathbb{C}, \mathbb{C}^*, E$};
\node [elli, above right,  xshift=4.5em, yshift=4em ] (1, 1) 
{(ii): Quantum Groups: $Y_{\hbar}(\g), U_q(L\g), E_{\hbar, \tau}(\g)$};
 \draw[thick][->] (-1.1, 0.8) -- (-2, 1.5);
  \draw[thick][->] (-2.1, 1.4) -- (-1.2, 0.7);
 \draw[dashed, thick][->] (1.1, 0.8) -- (2, 1.5);
  \draw[thick][->](2.1, 1.4) -- (1.2, 0.7);
   \draw[thick][->] (-0.8, 2.1) -- (0.8, 2.1);
  \draw[thick][->] (0.8, 1.9) --(-0.8, 1.9);
\end{tikzpicture}
\]
Here, the correspondence (i)$\leftrightarrow$(iii) is well known in algebraic topology, goes back to the work of Quillen, Hirzebruch, et al. Similar correspondence also exists in algebraic geometry thanks to the oriented cohomology theories (OCT) of Levine and Morel \cite{LM}.
The algebraic OCT associated to $\mathbb{C}, \mathbb{C}^*$ and $E$ are, respectively, the intersection theory  (Chow groups) $\text{CH}$, the algebraic K-theory $K$ and the algebraic elliptic cohomology $\Ell$. 

The correspondence (i)$\leftrightarrow$(ii) was introduced to the mathematical community by Drinfeld \cite{Dr}. Roughly speaking, the quantum group in (ii) quantizes of the Lie algebra of maps from the algebraic group in (i) to a Lie algebra $\g$. The quantization is obtained from  
the solutions to the quantum Yang-Baxter equation
\begin{equation}
R_{12}(u)R_{13}(u+v) R_{23}(v)=R_{23}(v) R_{13}(u+v) R_{12}(u). 
\tag{QYBE}
\end{equation}
The Yangian $Y_{\hbar}(\g)$, quantum loop algebra $U_q(L\g)$, and elliptic quantum group $E_{\hbar, \tau}(\g)$ are respectively obtained from the rational, trigonometric and elliptic solutions of the QYBE. 
There is a dynamical elliptic quantum group associated to a general symmetrizable Kac-Moody Lie algebra, which is obtained from solutions to the dynamical Yang-Baxter equation.

The correspondence (ii)$\leftrightarrow$(iii) is the main subject of this note, and is originated from  the work of Nakajima \cite{Nak99}, Varagnolo  \cite{Va}, Maulik-Okounkov \cite{MO}, and many others. Without going to the details, the quantum group in (ii) acts on the corresponding  oriented cohomology of the Nakajima quiver varieties recalled below. 

Let $Q=(I, H)$ be a quiver, with $I$ being the set of vertices, and $H$ being the set of arrows. 
Let $Q^{\heartsuit}$ be the framed quiver, schematically illustrated below. For any dimension vector $(\vec{v}, \vec{w})$ of $Q^{\heartsuit}$, we have the Nakajima quiver variety $\mathfrak{M}(\vec{v}, \vec{w})$.
$$
\begin{matrix}
\xymatrix{
\vec{v}: \bigcirc
\ar@/^/[r]
\ar@/_/[r]
&
\bigcirc
\ar[r]
 &
\bigcirc
\ar@(ul,ur)
 }
 \\  \\
 \text{Quiver $Q$}
 \end{matrix}
 \,\ \,\ \,\ \,\  \,\ \,\ 
 \begin{matrix}
\xymatrix{
*+[l]{\vec{v}: \bigcirc}
\ar@/^/[r]
\ar@/_/[r]
\ar[d]&
\bigcirc
\ar[r]
 \ar[d]&
\bigcirc
\ar@(ul,ur)
\ar[d] \\
*+[l]{\vec{w}: \square}  &
\square &
\square 
}
\\ \\
 \text{Framed quiver $
Q^{\heartsuit}$}
\end{matrix}
 $$
Denote $\mathfrak{M}(\vec{w})=\coprod_{\vec{v}\in \mathbb{N}^I}\mathfrak{M}(\vec{v}, \vec{w})$. We follow the terminologies and notations in \cite{Nak99} on quiver varieties, hence refer the readers to {\it loc. cit.} for the details. Nevertheless, we illustrate by the following examples, and fix some conventions in \S~\ref{sec:Construction}.

\begin{example}
\leavevmode
\begin{enumerate}
\item Let $Q$ be the quiver with one vertex, no arrows. Let $(\vec{v}, \vec{w})=(r, n)$, with $0\leq r\leq n$. Then,  
$\mathfrak{M}(r, n)\cong T^*\Gr(r, n)$, the cotangent bundle of the Grassmannian. 
\item Let $Q$ be the Jordan quiver with one vertex, and one self-loop. Let $(\vec{v}, \vec{w})=(n, 1)$, with $n\in \mathbb{N}$. Then, 
$\mathfrak{M}(n, 1)\cong \Hilb^n(\mathbb{C}^2)$, the Hilbert scheme of $n$-points on $\mathbb{C}^2$.
\end{enumerate}

\[
\begin{matrix}
\xymatrix{
*+[l]{r: \bigcirc} \ar[d]
 \\
*+[l]{n: \square} }\\
\text{Example (1)}
\end{matrix}
\,\ \,\ \,\ \,\ 
\begin{matrix}
\xymatrix{
*+[l]{n: \bigcirc}
\ar@(ul,ur)\ar[d] 
 \\
*+[l]{1: \square} }
\\
\text{Example (2)}
\end{matrix}
\]
\end{example}
When $Q$ has no edge-loops, Nakajima in \cite{Nak99} constructed an action of $U_q(L\g)$ on the equivariant K-theory of $\mathfrak{M}(\vec{w})$, with $\fg$ being the symmetric Kac-Moody Lie algebra associated to $Q$. Varagnolo in  \cite{Va} constructed an action of $Y_{\hbar}(\g)$ on the equivariant cohomology of the Nakajima quiver varieties. For a general quiver $Q$, Maulik-Okounkov in \cite{MO} constructed a bigger Yangian $\mathbb{Y}_{\MO}$ acting on the equivariant cohomology of $\mathfrak{M}(\vec{w})$ using a geometric R-matrix formalism. See \cite{OS, AO} for the geometric R-matrix construction in the trigonometric (K-theoretic stable envelope), and elliptic case (elliptic stable envelope). 

The goal of the present notes is to explain a direct construction of the correspondence from (iii) to (ii) above,   using cohomological Hall algebra (CoHA) following  \cite{YZ1}. Most of the constructions are direct generalizations of Schiffmann and Vasserot \cite{SV2}. Closely related is the CoHA of Kontsevich-Soibelman \cite{KS}, defined to be the critical cohomology (cohomology valued in a vanishing cycle) of the moduli space of representations of a quiver with potential. A relation between the CoHA in the present note and the CoHA from \cite{KS} is explained in \cite{YZ16}.

This approach of studying quantum groups has the following advantages.
\begin{itemize}
\item The construction works for any oriented cohomology theory, beyond  $\text{CH}, K, \Ell$. 
One interesting example is the Morava K-theory. The  new quantum groups obtained via this construction are expected to be related the Lusztig's character formulas \cite[\S 6]{YZ1}. 
\item In the case of $\Ell$, the construction gives a sheafified elliptic quantum group, as well as an action of $E_{\hbar, \tau}(\g)$ on the equivariant elliptic cohomology of Nakajima quiver varieties, as will be explained in \S\ref{sec:3}. 
\end{itemize}

\section{Construction of the cohomological Hall algebras}
\label{sec:Construction}
For illustration purpose,  in this section we take the OCT to be the intersection theory $\CH$. Most  statements have counterparts in an arbitrary oriented cohomology theory. 
\subsection{The theorem}
Let $Q=(I, H)$ be an arbitrary quiver. Denote by $\g_Q$ the corresponding symmetric Kac-Moody Lie algebra associated to $Q$. 
The preprojective algebra, denoted by $\Pi_Q$, is the quotient of the path algebra $\C( Q\cup Q^{op})$ by the ideal generated by the relation $\sum_{x\in H}[x, x^{*}]=0$, 
where $x^*\in H^{\op}$ is the reversed arrow of $x\in H$. Fix a dimension vector $\vec{v}  \in \N^I$, let $\Rep(Q, \vec{v})$ be the affine space parametrizing representations of the path algebra $\bbC Q$ of dimension $\vec{v}$, and let
$\Rep(\Pi_Q, \vec{v})$ be the affine algebraic variety parametrizing representations of $\Pi_Q$ of dimension $\vec{v}$, with an action of $\GL_{\vec{v}}:=\prod_{i\in I} \GL_{v^i}$. Here $\vec{v}=(v^i)_{i\in I}$.
Let $\cRep(\Pi_Q, \vec{v}):=\Rep(\Pi_Q, \vec{v})/\GL_{\vec{v}}$ be the quotient stack. 

\begin{example}
Let $Q$ be the Jordan quiver: $\xymatrix@C=0.5em{
{\bigcirc}
\ar@(ul,ur)-|{x}}$.  The preprojective algebra $\Pi_Q=\C[x, x^*]$ is the free polynomial ring in two variables. 
We have
\[
\cRep(\Pi_Q, n)=\{(A, B)\in (\mathfrak{gl}_n)^2 \mid [A, B]=0\} /\GL_{n}.
\]
\end{example}
Consider the graded vector space
\begin{equation}
\label{eq:calP}
\calP(\CH, Q):=\bigoplus_{\vec{v}\in \N^I} \CH_{\C^*}(\cRep(\Pi_Q, \vec{v}))=
\bigoplus_{\vec{v}\in \N^I} \CH_{\GL_{\vec{v}}\times \C^*}(\Rep(\Pi_Q, \vec{v})). 
\end{equation}
The torus $\C^*$ acts on $\Rep(\Pi_Q, \vec{v})$ the same way as in \cite[(2.7.1) and (2.7.2)]{Nak99}. More explicitly,  
let $a$ be the number of arrows in $Q$ from vertex $i$ to $j$; We enumerate these arrows as  $h_1, \dots, h_a$. The corresponding  reversed arrows in $Q^{\op}$ are enumerated as $h_1^*, \cdots, h_a^*$. We define an action of $\C^*$ on $\Rep(\Pi_Q, \vec{v})$ in the following way.
For $t\in \C^*$ and $(B_p, B_p^*) \in \Rep(\Pi_Q, \vec{v})$ with $h_p \in H$, we define 
\[
 t\cdot B_p:=t^{a+2-2p} B_p, \,\  t\cdot B_p^*:=t^{-a+2p} B_p^*
 \] 

\begin{thm}\cite[Yang-Zhao]{YZ1, YZ2} 
\label{thm:YZ1}
\begin{enumerate}
\item 
The vector space $\calP(\CH)$ is naturally endowed with a product $\star$ and a coproduct $\Delta$, making it a bialgebra. 
\item 
On $D(\calP):=\calP\otimes \calP$, there is a bialgebra structure obtained as a  Drinfeld double of $\calP(\CH)$. 
For any $\vec{w}\in \mathbb{N}^I$, the algebra $D(\calP)$ acts on $\CH_{\GL_{\vec{w}}}(\mathfrak{M}(\vec{w}))$. 
\item 
Assume $Q$ has no edge loops. There is a certain  {\it spherical subalgebra}  $D(\calP^{\sph})\subseteq D(\calP(\CH))$, such that
\[
D(\calP^{\sph})\cong Y_\hbar(\g_{Q}). 
\]
Furthermore, the induced Yangian action from (2) is compatible with the action constructed by Varagnolo in \cite{Va}, 
and, in the case when $Q$ is of $ADE$ type,  the action of Maulik-Okounkov in \cite{MO}. 
\end{enumerate}
\end{thm}
\begin{remark}
The construction of the Drinfeld double involves adding a Cartan subalgebra and defining a bialgebra pairing, as can be found in detail in \cite[\S3]{YZ2}.  The Cartan subalgebra can alternatively be replaced by a symmetric monoidal structure as in \cite[\S5]{YZell}.
The definition of the spherical subalgebra can be found in \cite[\S3.2]{YZ2}.
\end{remark}

\subsection{Constructions}
The Hall multiplication $\star$ of $\calP(\CH)$ is defined using the following correspondence. 
\[
\xymatrix@R=0.2em{
&\mathfrak{E}\text{xt}\ar[ld]_(0.3){\phi}\ar[rd]^(0.3){\psi}&\\
\cRep(\Pi_Q, \vec{v_1})\times \cRep(\Pi_Q, \vec{v_2})&&\cRep(\Pi_Q, \vec{v_1}+\vec{v_2})
}
\]
where $\mathfrak{E}\text{xt}$ is the moduli space of extensions $\{0\to V_1 \to V\to V_2\to 0 \mid \dim(V_i)=\vec{v_i}, i=1, 2\}$. 
The map $\phi: (0\to V_1 \to V\to V_2\to 0)\mapsto (V_1, V_2)$ is smooth, and $\psi: (0\to V_1 \to V\to V_2\to 0)\mapsto V$ is proper. The Hall multiplication $\star$ is defined to be
\[
\star=\psi_* \circ \phi^*. 
\]
Here the stacks $\cRep(\Pi_Q, \vec{v})$ for $\vec{v}\in\bbN^I$ are endowed with obstruction theories obtained from the embeddings $\cRep(\Pi_Q, \vec{v})\inj T^*Rep(Q, \vec{v})/\GL_{\vec{v}}$, and $\mathfrak{E}\text{xt}$ has a similar obstruction theory described in detail in \cite[\S 4.1]{YZ1}. Similar for the construction of the action below.

Now we explain the action in Theorem \ref{thm:YZ1} (2).
Let $\cRep^{\Frame}(\Pi_Q, \vec{v}, \vec{w})$ be the moduli space of framed representations of $\Pi_Q$ with dimension vector $(\vec{v}, \vec{w})$, which is constructed as a quotient stack $\Rep^{\Frame}(\Pi_Q, \vec{v}, \vec{w})/\GL_{\vec{v}}$. 
Imposing a suitable semistability condition, explained in detail in \cite{Nak99}, we get an open subset $\cRep^{\Frame, ss}(\Pi_Q, \vec{v}, \vec{w})\subset \cRep^{\Frame}(\Pi_Q, \vec{v}, \vec{w})$. There is an isomorphism 
\[
\CH(\cRep^{\Frame, ss}(\Pi_Q, \vec{v}, \vec{w}))=\CH_{\GL_{\vec{w}}}(\mathfrak{M}(\vec{v}, \vec{w})).
\] 

We have the following general correspondence  \cite[\S 4.1]{YZ1}:
\begin{equation}\label{eq:corr}
\xymatrix@R=0.2em@C=0.2em{
&\mathfrak{E}\text{xt}^{\Frame} \ar[ld]_(0.3){\overline{\phi}}\ar[rd]^(0.3){\overline{\psi}}&\\
\cRep^{\Frame}(\Pi_Q, \vec{v_1}, \vec{w_1})\times \cRep^{\Frame}(\Pi_Q, \vec{v_2}, \vec{w_2})&&\cRep^{\Frame}(\Pi_Q, \vec{v_1}+\vec{v_2}, \vec{w_1}+\vec{w_2})
}
\end{equation}
The action in Theorem \ref{thm:YZ1} (2) is defined as $(\overline{\psi}^{ss})_* \circ (\overline{\phi}^{ss})^*$ by taking $\vec{w_1}=\vec{0}$, and imposing a suitable semistability  condition on the correspondence \eqref{eq:corr}. 

\subsection{Shuffle algebra}
Notations as before, let $Q=(I, H)$ be a quiver. Following the proof of \cite[Proposition 3.4]{YZ1}, we have the following shuffle description of $(\overline{\psi})_* \circ (\overline{\phi})^*$ in \eqref{eq:corr}.

Let  $\SH$ be an $\bbN^I\times \bbN^I$-graded $\mathbb{C}[t_1,t_2]$-algebra. As a $\mathbb{C}[t_1,t_2]$-module, we have 
\[\SH=\bigoplus_{\vec{v}\in\bbN^I, \vec{w}\in\bbN^I}\SH_{\vec{v}, \vec{w}}, \,\ \text{where}\,\ \SH_{\vec{v}, \vec{w}}:=\mathbb{C}[t_1,t_2]\otimes \mathbb{C}
[ \lambda^i_s ]_{i\in I, s=1,\dots, v^i}^{\fS_{\vec{v}}}\otimes \mathbb{C}
[ z^j_t ]_{j\in I, t=1,\dots, w^j}^{\fS_{\vec{w}}} ,\]
here $\fS_{\vec{v}}=\prod_{i\in I} \fS_{v^i}$ is the product of symmetric groups, and $\fS_{\vec{v}}$ naturally acts on the variables $\{ \lambda^i_s \}_{i\in I, s=1,\dots, v^i}$ by permutation. 
For any $(\vec{v_1}, \vec{w_1})$ and $(\vec{v_2}, \vec{w_2})\in \bbN^I\times \bbN^I$, 
we consider $\SH_{\vec{v_1}, \vec{w_1}}\otimes_{\mathbb{C}[t_1,t_2]} \SH_{\vec{v_2}, \vec{w_2}}$ as a subalgebra of $
\mathbb{C}[t_1,t_2][\lambda^i_s, z^j_t]_{\big\{\begin{smallmatrix}i\in I, s=1,\dots, ({v_1^i}+{v_2^i}), \\ j\in I, t=1,\dots, ({w_1^j}+{w_2^j}) \end{smallmatrix}\big\}}$ 
by sending $(\lambda'^i_s, z'^j_t) \in \SH_{\vec{v_1}, \vec{w_1}} $ to $(\lambda^i_s, z^j_t)$, and 
$(\lambda''^i_s, z''^j_t) \in \SH_{\vec{v_2}, \vec{w_2}} $ to $(\lambda^i_{s+v_1^i}, z''^j_{t+w_1^j})$.

We define the shuffle product $\SH_{\vec{v_1}, \vec{w_1}}\otimes_{\mathbb{C}[t_1,t_2]} \SH_{\vec{v_2}, \vec{w_2}}\to \SH_{\vec{v_1}+\vec{v_2}, \vec{w_1}+\vec{w_2}}$,
\begin{align}
&f(\lambda_{\vec{v_1}} , z_{\vec{w_1}})\otimes g(\lambda_{\vec{v_2}} , z_{\vec{w_2}})\mapsto
\sum_{\sigma\in \Sh(\vec{v_1}, \vec{v_2})\times \Sh(\vec{w_1}, \vec{w_2})} \sigma\Big(f(\lambda'_{\vec{v_1}} , z'_{\vec{w_1}})\cdot g(\lambda''_{\vec{v_2}} , z''_{\vec{w_2}})\cdot \fac_{\vec{v_1}+\vec{v_2},\vec{w_1}+\vec{w_2}}\Big), \label{eq:shuffle}
\end{align}
with $\fac_{\vec{v_1}+\vec{v_2},\vec{w_1}+\vec{w_2}}$ specified as follows. 

Let
\begin{equation}\label{equ:fac1}
\fac_1:=\prod_{i\in I}\prod_{s=1}^{v_1^i}
\prod_{t=1}^{v_2^i}\frac{\lambda\rq{}^i_s-\lambda\rq{}\rq{}^i_t+t_1+t_2}{\lambda\rq{}\rq{}^i_t-\lambda\rq{}^i_s}. 
\end{equation}
Let $m:H\coprod H^{\op}\to \bbZ$ be a function, which for each $h\in H$ provides two integers $m_h$ and $m_{h^*}$. We define the torus  $T=(\mathbb{C}^*)^2$ action on $\Rep(Q\cup Q^{\op})$ according to the function $m$, satisfying some technical conditions spelled out in \cite[Assumption~1.1]{YZ2}. The $T$-equivariant variables are denoted by $t_1,t_2$.
Define
\begin{align}
\label{equ:fac2}
\fac_2:=&\prod_{h\in H}\Big(
\prod_{s=1}^{v_1^{\out(h)}}
\prod_{t=1}^{v_2^{\inc(h)}}
(\lambda_t^{'' \inc(h)}-\lambda_s^{'\out(h)}+ m_h t_1)
\prod_{s=1}^{v_1^{\inc(h)}}
\prod_{t=1}^{v_2^{\out(h)}}
(\lambda_t^{''\out(h)}-\lambda_s^{'\inc(h)}+m_{h^*} t_2)
\Big) \notag\\
&\cdot \prod_{i\in I} 
\Big(
\prod_{s=1}^{v_1^{i}}
\prod_{t=1}^{w_2^{i}}
(z_t^{''i}-\lambda_s^{'i}+ t_1)
\prod_{s=1}^{w_1^{i}}
\prod_{t=1}^{v_2^{i}}
(\lambda_t^{'' i}-z_s^{'i}+  t_2)
\Big)
\end{align}
Let
\begin{equation}\label{eq:fac formula}
\fac_{\vec{v_1}+\vec{v_2},\vec{w_1}+\vec{w_2}}:=\fac_1 \cdot \fac_2. 
\end{equation}
 \begin{prop}
 \label{prop:shuffle}
Under the identification 
\[
\CH_{\GL_{\vec{v}}\times \GL_{\vec{w}}\times T}(\Rep^{\Frame}(\Pi_Q, \vec{v}, \vec{w}))\cong
\mathbb{C}[t_1,t_2]\otimes \mathbb{C}[ \lambda^i_s]_{i\in I, s=1,\dots, v^i}^{\fS_{\vec{v}}} \otimes \mathbb{C}[z^{j}_{t}]_{j\in I, s=1,\dots, w^i}^{\fS_{\vec{w}}}=:
\SH_{ \vec{v}, \vec{w}},
\] 
the map $(\overline{\psi})_* \circ (\overline{\phi})^*$ is equal to the multiplication \eqref{eq:shuffle} of the shuffle algebra $\SH=\bigoplus_{\vec{v}\in\bbN^I, \vec{w}\in\bbN^I}\SH_{\vec{v}, \vec{w}}$.
\end{prop}
\begin{proof}
The proof follows from the same proof as \cite[Proposition 3.4]{YZ1} replacing the quiver $Q$ by the framed quiver $Q^{\heartsuit}$. 
\end{proof}
\begin{remark}
\begin{enumerate}
\item 
For an arbitrary cohomology theory, Proposition \ref{prop:shuffle} is still true when $A+B$ is replaced by $A+_{F}B$ in the formula \eqref{eq:fac formula} of $\fac_{\vec{v_1}+\vec{v_2},\vec{w_1}+\vec{w_2}}$, where $F$ is the  formal group law associated to this cohomology theory. 
\item Restricting to the open subset $\cRep^{\Frame,ss}(\Pi_Q, \vec{v})$ of $\cRep^{\Frame}(\Pi_Q, \vec{v})$ induces a surjective map $\SH_{\vec{v}, \vec{w}}\to \CH_{\GL_{\vec{w}}}(\mathfrak{M}(\vec{v}, \vec{w}))$, for $\vec{v}\in \N^I, \vec{w}\in \N^I$. The surjectivity follows from \cite{MN}. This map is compatible with the shuffle product of the left hand side, and the multiplication on the right hand side induced from \eqref{eq:corr}.
\end{enumerate}
\end{remark}

\subsection{Drinfeld currents}
Let $\vec{v}=e_i$ be the dimension vector valued 1 at vertex $i\in I$ and zero otherwise. Then, $\calP_{e_i}=\CH_{\GL_{e_i}\times \C^*}(\Rep(\Pi_Q, e_i))=\C[\hbar][x_i]$. 
Let $(x_i)^k \in \calP_{e_i}$, $k\in \N$, be the natural monomial basis of  $\calP_{e_i}$. One can form the Drinfeld currents
\[
\mathfrak{X}^+_i(u):=\sum_{k\in N} (x_i)^k u^{-k-1}, \,\ i\in I. 
\]
By Theorem \ref{thm:YZ1}, the generating series $\{\mathfrak{X}^+_i(u)\mid i\in I\}$ satisfy the relations of $Y_{\hbar}^+(\fg)[[u]]$ \cite[\S~7]{YZ1}.

\section{Sheafified elliptic quantum groups}
\label{sec:3}
In this section, we applying the equivariant elliptic cohomology to the construction in \S\ref{sec:Construction}. It gives a sheafified elliptic quantum group, as well as its action on $\mathfrak{M}(\vec{w})$. 

\subsection{Equivariant elliptic cohomology}
\label{sec:ell cohomology}
There is a sheaf-theoretic definition of the equivariant elliptic cohomology theory $\mathcal{E}ll_{G}$ in \cite{GKV} by Ginzburg-Kapranov-Vasserot. It was investigated by many others later on, including Ando \cite{And00, And03}, Chen \cite{Chen10}, Gepner\cite{Gep}, Goerss-Hopkins \cite{GH}, Lurie \cite{Lurie}. 

Let $\mathcal X_G$ be the moduli scheme of semisimple semistable degree 0 $G$-bundles over an elliptic curve. 
For a $G$-variety $X$, the $G$-equivariant elliptic cohomology $\mathcal{E}ll_{G}(X)$ of $X$ is a quasi-coherent sheaf of $\mathcal{O}_{\mathcal X_G}$-module, satisfying certain axioms.
In particular, $\mathcal{E}ll_{G}(\pt)\cong \mathcal{O}_{\mathcal X_G}$.
\begin{example} 
\begin{enumerate}
  \item
Let $G=S^1$, then $\mathcal{E}ll_{S^1}(X)$ is a coherent sheaf on $\text{Pic}(E)\cong E$.
This fact leads to the following patten. 
\begin{itemize}
\item
$\CH_{S^1}(X)$ is a module over $\CH_{S^1}(\pt)=\mathcal{O}_{\C}$.
\item
$K_{S^1}(X)$ is a module over $K_{S^1}(\pt)=\mathcal{O}_{\C^*}$.
\item
$\mathcal{E}ll_{S^1}(X)$ is a module over $\mathcal{E}ll_{S^1}(\pt)=\mathcal{O}_{E}$.
\end{itemize}
\item
Let $G=\GL_n$, then $\mathcal{E}ll^*_{\GL_n}(X)$ is a coherent sheaf over $E^{(n)}=E^{n}/\mathfrak{S}_n$.  
\end{enumerate}
\end{example}
There is a subtlety for pushforward in the equivariant elliptic cohomology theory. 
Let $f: X\to Y$ be a proper, $G$-equivariant homomorphism. 
The pushforward $f_*$ is the following map
\[
f_*: \mathcal{E}ll_{G}(\Th(f)) \to \mathcal{E}ll_{G}(Y), 
\]
where $\mathcal{E}ll_{G}(\Th(f))$, depending on the Thom bundle of the relative tangent bundle of $f$,  is a rank 1 locally free module over $\mathcal{E}ll_{G}(X)$. 
The appearance of this twist can be illustrated in the following simple example.  The general case is discussed in detail in \cite[\S~2]{GKV} and \cite[\S~2]{YZell}.
\begin{example}
Let $f: \{0\} \inj \C$ be the inclusion. The torus $S^1$ acts on $\C$ by scalar multiplication. Denote by $D$ the disc around $0$. We have the Thom space $\Th:=D/S^1$. 
There is an exact sequence
\[
0\to \mathcal{E}ll_{S^1}(\Th) \to \mathcal{E}ll_{S^1}(D) \to \mathcal{E}ll_{S^1}(S^1)\to 0. 
\]
As $\mathcal{E}ll_{S^1}(D)\cong \calO_E$, since $D$ is contractible, and $\mathcal{E}ll_{S^1}(S^1)$ is the skyscraper sheaf $\C_0$ at $0$, we have the isomorphism
$ \mathcal{E}ll_{S^1}(\Th)\cong \calO(-\{0\})$. 
\end{example}
\subsection{The sheafified elliptic quantum group}
Recall the elliptic cohomological Hall algebra (see \eqref{eq:calP}) is defined as
\[
\calP(\Ell, Q):=\bigoplus_{\vec{v}\in \N^I} \Ell_{\GL_{\vec{v}}\times \C^*}(\Rep(\Pi_Q, \vec{v})). 
\]
By the discussion in \S\ref{sec:ell cohomology}, $\calP(\Ell, Q)$ is a sheaf on 
\[
\calH_{E\times I} \times E_{\hbar}:=\coprod_{\{\vec{v}=(v^i)_{i\in I} \in \N^I\}} (E^{ (v^1)}\times E^{ (v^2)}\times \cdots \times E^{(v^n)})\times E_{\hbar}, 
\] 
where $\hbar$ comes from the $\C^*$-action, and $\calH_{E\times I}$ is the moduli space of $I$-colored points on $E$. 
\begin{center}
\includegraphics[scale = 0.14]{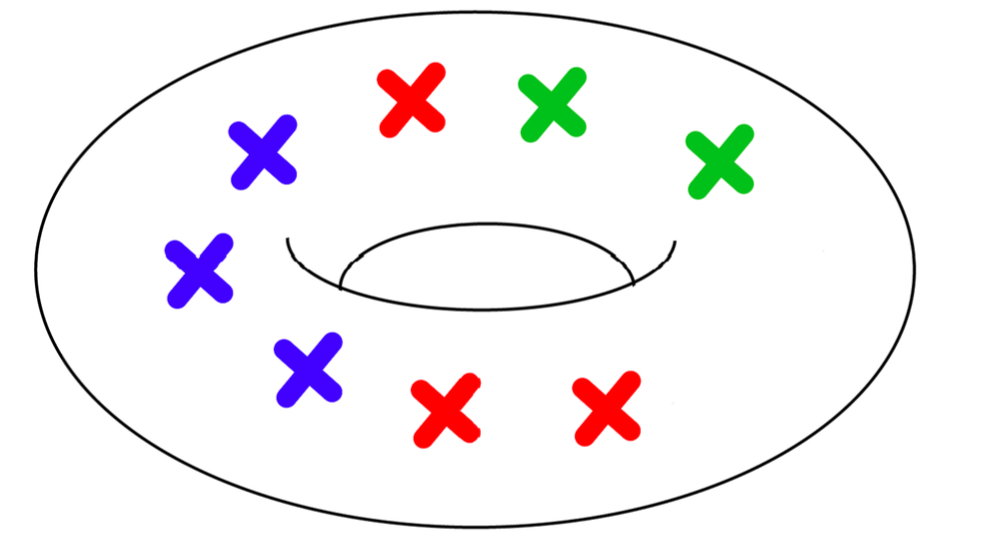}
\end{center}
Due to the subtlety of pushing-forward in the elliptic cohomology, there is no product on $\calP(\Ell, Q)$ in the usual sense. We illustrate the structure of the Hall multiplication $\star$ of $\calP(\Ell, Q)$ in the following example. 
\begin{example}
Let $Q$ be the quiver $\bigcirc$ with one vertex, no arrows. In this case, we have $\g_Q=\mathfrak{sl}_2$. The elliptic CoHA $\calP(\mathcal{E}ll, \mathfrak{sl}_2)$ associated to the Lie algebra $\mathfrak{sl}_2$ consists of:
\begin{itemize}
\item  A  coherent sheaf $\calP(\mathcal{E}ll, \mathfrak{sl}_2)=(\calP_n)_{n\in \N}$ on $\calH_E\times E_{\hbar}=\coprod_{n\in \N}E^{(n)}\times E_{\hbar}$. 
\item 
For any $n, m\in \mathbb{N}$, a morphism of sheaves on $E^{(n+m)}\times E_{\hbar}$: \[
\star: 
(\Sigma_{n, m})_{*}\big( (\calP_{n} \boxtimes 
\calP_{m})\otimes \mathcal L_{n, m}\big)
\to \calP_{n+m},\]
where $\Sigma_{n, m}$ the symmetrization map 
$
E^{(n)}\times E^{ (m)}\to E^{(n+m)}
$, and $\mathcal L_{n, m}$ is some fixed line bundle on $E^{(n)}\times E^{(m)}\times E_{\hbar}$, depending on the parameter $\hbar$. 
\item
The above morphisms are associative in the obvious sense.
\Omit{
 that the two different ways to obtain the map
\[
\pi_{v_1, v_2, v_3*}\Big( \mathcal{E}ll_{G_{v_1}}(\mu_{v_1}^{-1}(0)) \boxtimes 
\mathcal{E}ll_{G_{v_2}}(\mu_{v_2}^{-1}(0))
\boxtimes 
\mathcal{E}ll_{G_{v_3}}(\mu_{v_3}^{-1}(0))\Big)
\to
\mathcal{E}ll_{G_{v}}(\mu_{v}^{-1}(0))
\]
are the same, for $v=v_1+v_2+v_3$.}
\end{itemize}
\end{example}
Let $\mathcal{C}$ be the category of coherent sheaves on 
$\mathcal{H}_{E\times I}\times E_{\hbar}$. Motivated by the Hall multiplication $\star$ of $\calP(\Ell, Q)$, we define a tensor structure 
$\otimes_{\hbar}$ on $\mathcal{C}$: for $\{\mathcal{F}\}, \{\mathcal{G}\}\in \text{Obj}(\mathcal{C})$, $\mathcal{F}\otimes_{\hbar} \mathcal{G}$ is defined as
\begin{equation}\label{equ:tensor on C}
(\mathcal{F}\otimes_{{\hbar}} \mathcal{G})_{\vec{v}}:=\bigoplus_{\vec{v}_1+\vec{v}_2=\vec{v}}
(\Sigma_{\vec{v}_1, \vec{v}_2})_{*}\big( (\mathcal{F}_{\vec{v}_1} \boxtimes 
 \mathcal{G}_{\vec{v}_1})\otimes \mathcal L_{\vec{v}_1, \vec{v}_2}\big). 
\end{equation}
\begin{thm}\label{conj:ell}\cite[Yang-Zhao]{YZell}
\leavevmode
\begin{enumerate}
\item The category $(\calC, \otimes_{\hbar})$ is a symmetric monoidal category, with the braiding given by \cite[Theorem~3.3]{YZell}. 
\item 
The elliptic CoHA $(\calP(\mathcal{E}ll, Q), \star, \Delta)$, endowed with the Hall multiplication $\star$, and coproduct $\Delta$, is a bialgebra object in $(\calC^{\loc}, \otimes_{\hbar})$. 
\item The Drinfeld double $D(\calP(\mathcal{E}ll, Q))$ of $\calP(\mathcal{E}ll, Q)$ acts on $\mathcal{E}ll_{G_{\vec{w}}}(\mathfrak{M}(\vec{w}))$, for any $\vec{w}\in \N^I$. 
\item After taking a certain space of meromorphic sections $\Gamma_{\rat}$, the bialgebra $\Gamma_{\rat}(D(\calP^{\sph}_\lambda(\mathcal{E}ll, Q)))$ becomes the elliptic quantum group given by the dynamical elliptic R-matrix of Felder \cite{F1}, Gautam-Toledano Laredo \cite{GTL15}. 
\end{enumerate}
\end{thm}
\begin{remark}
\begin{enumerate}
\item
In Theorem~\ref{conj:ell}(4), a version of the sheafified elliptic quantum group with dynamical twist $D(\calP^{\sph}_\lambda(\mathcal{E}ll, Q))$ is needed in order to recover the R-matrix of Felder-Gautam-Toledano Laredo. This twist is explained in detail in \cite[\S~10.2]{YZell}. 
Below we illustrate the flavour of this dynamical twist in  \S~\ref{subsec:ellCurrents}.

In particular, the abelian category of representations of $D(\calP(\mathcal{E}ll, Q))$ and  $D(\calP^{\sph}_\lambda(\mathcal{E}ll, Q))$  are both well-defined (see \cite[\S~9]{YZell} for the details). Furthermore, it is proved in {\it loc. cit.} that these two representation categories are equivalent as abelian categories. 
\item The details of the space of meromorphic sections are explained in \cite[\S~6]{YZell}.
\item Based on the above theorem, we define the sheafified elliptic quantum group to be the Drinfeld double of $\calP^{\sph}(\mathcal{E}ll, Q)$. 
\end{enumerate}
\end{remark}

\subsection{The shuffle formulas}
\label{sec:shuffle elliptic}
The shuffle formula of the elliptic quantum group is given by \eqref{eq:shuffle}, with the factor $A+B$ replaced by $\vartheta(A+B)$. The shuffle formula gives an explicit description of the elliptic quantum group, as well as its action on the Nakajima quiver varieties. In this section, we illustrate the shuffle description of the elliptic quantum group $E_{\tau, \hbar}(\mathfrak{sl}_2)$. Furthermore, we show  \eqref{eq:shuffle} applied to special cases coincides with the shuffle formulas in \cite[Proposition 3.6]{FRV} and \cite[Definition 5.9]{K}. 

When $\mathfrak{g}=\mathfrak{sl}_2$, the corresponding quiver is $\bigcirc$, with one vertex, no arrows.  Let $\SH_{\vec{w}=0}:=\calO_{\calH_{E}\times E_{\hbar}}=(\bigoplus_{n\in \N} \calO_{E^{(n)}})\boxtimes \calO_{E_{\hbar}}$. 
For any local sections $f\in \calO_{E^{(n)}}, g\in \calO_{E^{(m)}}$, by \eqref{eq:shuffle} and \eqref{eq:fac formula}, we have
\[
f\star g=\sum_{\sigma\in \Sh(n, m)} \sigma \left(f g \prod_{1\leq s\leq n, n+1\leq t\leq n+m} \frac{\vartheta(x_s-x_t+\hbar)}{\vartheta(x_s-x_t)}\right), 
\]
where $f\star g\in  \calO_{E^{(n+m)}}$ and $\Sh(n, m)$ consists of $(n, m)$-shuffles, i.e., permutations of $\{1,\cdots ,n+m\}$ that preserve the relative order of 
$\{1,\cdots ,n\}$ and $\{ n+1, \cdots, n+m\}$. The elliptic quantum group $E_{\tau, \hbar}^+(\mathfrak{sl}_2)$ is a subalgebra of $(\SH_{\vec{w}=0}, \star)$. 

In the following examples we consider $\SH$ with general $\vec{w}$.
\begin{example}
Assume the quiver is $\bigcirc$, with one vertex, no arrows. 
Let $\vec{v_1}=k'$, $\vec{v_2}=k''$, and $\vec{w_1}=n'$, $\vec{w_2}=n''$. 
Choose $t_1=\hbar$, and $t_2=0$. Applying the formula \eqref{eq:fac formula} to this case, we have
\begin{align*}
\fac_{k'+k'', n'+n''}&=\prod_{s=1}^{k'}
\prod_{t=k'+1}^{k'+k''}\frac{\vartheta(\lambda_s-\lambda_t+\hbar)}{\vartheta(\lambda_t-\lambda_s)}  \cdot 
\prod_{s=1}^{k'}
\prod_{t=n'+1}^{n'+n''}
\vartheta(z_t-\lambda_s+\hbar)
\prod_{s=1}^{n'}
\prod_{t=k'+1}^{k'+k''}
\vartheta(\lambda_t-z_s), 
\end{align*}
where $\vartheta(z)$ is the odd Jacobi theta function, normalized such that $\vartheta'(0)=1$. 
This is exactly the same formula as \cite[Proposition 3.6]{FRV}. \end{example}

\begin{example}
When $\mathfrak{g}=\mathfrak{sl}_{N}$, the corresponding quiver is 
$\xymatrix{ \bigcirc \ar[r]&\bigcirc \ar[r] &  \cdots  \ar[r]&\bigcirc}$, with $N-1$ vertices, and $N-2$ arrows. 
Consider the framed quiver
\[
Q^{\heartsuit}: \xymatrix{
 \bigcirc \ar[r]&\bigcirc \ar[r] & \cdots \ar[r]&\bigcirc 
\ar[d]  \\
&&& \square }
\]
Label the vertices of $Q^{\heartsuit}$ by $\{1, 2, \cdots, N-1, N\}$. Let $\vec{v_1}=(v_1^{(l)})_{l=1, \cdots, N-1}$, $\vec{v_2}=(v_2^{(l)})_{l=1, \cdots, N-1}$ be a pair of dimension vectors, and take the framing to be $\vec{w_1}=(0, \cdots, 0, n)$, $\vec{w_2}=(0, \cdots, 0, m)$. 
To simplify the notations, let $v_1^{(N)}=n$, $v_2^{(N)}=m$, and denote the variables 
$z_{\vec{w_1}}$ by $\{\lambda_{s}^{'(N)}\}_{s=1, \cdots, v_1^{(N)}}$, and $z_{\vec{w_2}}$ by $\{\lambda_{t}^{''(N)}\}_{t=1, \cdots, v_2^{(N)}}$. 
Applying the formula \eqref{eq:fac formula} to this case, we then have
\begin{align*}
\fac_{\vec{v_1}+\vec{v_2},\vec{w_1}+\vec{w_2}}
=\prod_{l=1}^{N-1} &\Bigg(\prod_{s=1}^{v_1^{(l)}}
\prod_{t=1}^{v_2^{(l)}}\frac{\vartheta(\lambda\rq{}^{(l)}_s-\lambda\rq{}\rq{}^{(l)}_t+t_1+t_2)}{\vartheta(\lambda\rq{}\rq{}^{(l)}_t-\lambda\rq{}^{(l)}_s)} \\
&
\cdot \prod_{s=1}^{v_1^{(l)}}
\prod_{t=1}^{v_2^{(l+1)}}
\vartheta(\lambda_t^{'' (l+1)}-\lambda_s^{'(l)}+ t_1)
\prod_{t=1}^{v_1^{(l+1)}}
\prod_{s=1}^{v_2^{(l)}}
\vartheta(\lambda_s^{''(l)}-\lambda_t^{'(l+1)}+t_2)
\Bigg). 
\end{align*}

Following  \cite[\S 5 (5.3)]{K}, we denote by $H_{\vec{v}+\vec{w}}(\lambda_{\vec{v}}, z_{\vec{w}} )$ the following element 
\[
H_{\vec{v}+\vec{w}}(\lambda_{\vec{v}}, z_{\vec{w}})=
\prod_{l=1}^{N-1}  \prod_{s=1}^{v^{(l)}}
\prod_{t=1}^{v^{(l+1)}}
\vartheta(\lambda_t^{(l+1)}-\lambda_s^{(l)}+ t_1).\]
Define $H_{\text{cross}}=
\frac{ H_{\vec{v_1}+\vec{v_2}+\vec{w_1}+\vec{w_2}}(\lambda_{\vec{v_1}}\cup \lambda_{\vec{v_2}}, z_{\vec{w_1}}\cup z_{\vec{w_2}})}{
H_{\vec{v_1}+\vec{w_1}}(\lambda_{\vec{v_1}}, z_{\vec{w_1}})\cdot H_{\vec{v_2}+\vec{w_2}}(\lambda_{\vec{v_2}}, z_{\vec{w_2}})}$. We have
\begin{align}
H_{\text{cross}}=&\prod_{l=1}^{N-1}  \Bigg(
\prod_{s=1}^{v_1^{(l)}}
\prod_{t=1}^{v_2^{(l+1)}}
\vartheta(\lambda_t^{'' (l+1)}-\lambda_s^{'(l)}+ t_1)
\prod_{t=1}^{v_1^{(l+1)}}
\prod_{s=1}^{v_2^{(l)}}
\vartheta(\lambda_t^{' (l+1)}-\lambda_s^{''(l)}+ t_1)
\Bigg) \notag\\
=&
\prod_{l=1}^{N-1}  \Bigg(
\prod_{s=1}^{v_1^{(l)}}
\prod_{t=v_1^{(l+1)}+1}^{v_1^{(l+1)}+v_2^{(l+1)}}
\vartheta(\lambda_t^{ (l+1)}-\lambda_s^{(l)}+ t_1)
\prod_{t=1}^{v_1^{(l+1)}}
\prod_{s=v_1^{(l)}+1}^{v_1^{(l)}+v_2^{(l)}}
\vartheta(\lambda_t^{ (l+1)}-\lambda_s^{(l)}+ t_1)
\Bigg). \label{H cross term}
\end{align}
Divide $\fac_{\vec{v_1}+\vec{v_2},\vec{w_1}+\vec{w_2}}$ by $H_{\text{cross}}$,
we obtain
\begin{align*}
&\frac{\fac_{\vec{v_1}+\vec{v_2},\vec{w_1}+\vec{w_2}}}{H_{\text{cross}}}
=(-1)^{v_1^{(l+1)}+v_2^{(l)}}
\prod_{l=1}^{N-1} \Bigg(\prod_{s=1}^{v_1^{(l)}}
\prod_{t=1}^{v_2^{(l)}}\frac{\vartheta(\lambda\rq{}^{(l)}_s-\lambda\rq{}\rq{}^{(l)}_t+t_1+t_2)}{\vartheta(\lambda\rq{}\rq{}^{(l)}_t-\lambda\rq{}^{(l)}_s)} 
\cdot \prod_{t=1}^{v_1^{(l+1)}}
\prod_{s=1}^{v_2^{(l)}}
\frac{\vartheta(\lambda_t^{'(l+1)}-\lambda_s^{''(l)}-t_2)}{
\vartheta(\lambda_t^{' (l+1)}-\lambda_s^{''(l)}+ t_1)
}
\Bigg). 
\end{align*}
This coincides with the formula in \cite[Definition 5.9]{K} when $t_1=-1, t_2=0$. 

In other words, 
consider the map 
\begin{align*}
\bigoplus_{\vec{v}\in \N^I, \vec{w}\in \N^I} \SH_{\vec{v},\vec{w}} &\to \bigoplus_{\vec{v}\in \N^I, \vec{w}\in \N^I} (\SH_{\vec{v},\vec{w}})_{\loc}, \,\ \text{given by} \\f(\lambda_{\vec{v}}, z_{\vec{w}})&\mapsto \frac{ f(\lambda_{\vec{v}}, z_{\vec{w}})}{H_{\vec{v}+\vec{w}}(\lambda_{\vec{v}}, z_{\vec{w}})}.
\end{align*}
It intertwines the shuffle product \eqref{eq:shuffle} (with theta-functions), and the shuffle product of  \cite[Definition 5.9]{K}.
\end{example}

\begin{remark}
\begin{enumerate}
\item 
In the work of \cite{FRV} in the $\mathfrak{sl}_2$ case, and \cite{K} in the $\mathfrak{sl}_N$ case, the shuffle formulas are used to obtain an inductive formula for the elliptic weight functions. 
Proposition \ref{prop:shuffle} provides a way to define the elliptic weight functions associated to a general symmetric Kac-Moody Lie algebra $\g$. 
\item 
In the above cases for $\mathfrak{sl}_N$ and the special framing, the elliptic weight functions are expected to be related to the elliptic stable basis in \cite{AO} (see also \cite{FRV, K}). Therefore, it is reasonable to expect an expression of the elliptic stable basis in terms of the shuffle product \eqref{eq:shuffle} (with theta-functions) for general quiver varieties.
\end{enumerate}
\end{remark}

\subsection{Drinfeld currents}\label{subsec:ellCurrents}
We now explain the Drinfeld currents in the elliptic case. The choice of an elliptic curve $E$ gives rise to the dynamical elliptic quantum group. 

Let $\calM_{1, 2}$ be the open moduli space of 2 pointed genus 1 curves. 
We write a point in $\calM_{1, 2}$ as $(E_\tau, \lambda)$, where $E_\tau=\C/\Z\oplus\tau\Z$, and $\lambda$ gives a line bundle $\mathbb{L}_{\lambda}\in \text{Pic}(E_\tau)$. 
Let $E$ be the universal curve on $\calM_{1, 2}$. There is a Poincare line bundle $\mathbb{L}$ on $E$, which has a natural rational section
\[
\frac{\vartheta(z+\lambda)}{\vartheta(z)\vartheta(\lambda)}
\]
where $z$ is the coordinate of $E_{\tau}$, and $\vartheta(z)$ is the Jacobi odd theta function, normalized such that $\vartheta'(0)=1$.

We can twist the equivariant elliptic cohomology $\Ell_G$ by the Poincare line bundle $\mathbb{L}$. For each simple root $e_k$, after twisting, we have 
$\SH_{e_k}^{\mathbb{L}}=\calO_{E^{(e_k)}}\otimes \mathbb{L}$. A basis of the meromorphic sections $\Gamma_{\rat}(\SH_{e_k}^{\mathbb{L}})$ consists of 
$\left\{ g^{(i)}_{\lambda_k}(z_k):= \frac{\partial^i}{\partial z_k^i}\left( \frac{\vartheta(z_k+\lambda_k)}{\vartheta(z_k)\vartheta(\lambda_k) }\right)\right\}_{i\in \N}$. 

Consider $\lambda=(\lambda_k)_{k\in I}$, and let
\[
\mathfrak{X}_{k}^+(u, \lambda):=
\sum_{i=0}^{\infty} g^{(i)}_{\lambda_k}(z_k) u^{i}=\frac{\vartheta(z_k+\lambda_k+u)}{\vartheta(z_k+u)\vartheta(\lambda_k)} \in \Gamma_{\rat}(\SH^{\mathbb{L}})[[u]], \,\ k\in I. 
\]
Similarly, we define series $\mathfrak{X}_{k}^-(u, \lambda)$, $\Phi_k(u)$. The series 
$\mathfrak{X}_{k}^+(u, \lambda)$, $\mathfrak{X}_{k}^-(u, \lambda)$, and  $\Phi_{k}(u)$ satisfy the relations of the elliptic quantum group of Gautam-Toledano Laredo \cite{GTL15}. 

\section{Relation with the affine Grassmannian}
In this section, we explain one unexpected feature of the sheafified elliptic quantum group, namely, its relation with the global loop Grassmannians over an elliptic curve $E$. 
We assume the quiver $Q$=(I, H) is of type ADE in this section. 

We collect some facts from \cite{Mirk, Mir2}. Let $C$ be a curve. An $I$-colored \textit{local space} $Z$ over $C$, defined in \cite[Section 2]{Mirk}, see also \cite[Section 4.1]{Mir2}, is a space $Z\to \mathcal{H}_{C\times I}$ over the $I$-colored Hilbert scheme of points of $C$, together with a consistent system of isomorphisms 
\[
Z_{D}\times Z_{D'} \to Z_{D \sqcup D'}, \,\ \text{for $D, D'\in \calH_{C\times I}$ with $D\cap D'=\emptyset$.} 
\]
Similarly, for a coherent sheaf $\calF$ on $\calH_{C\times I}$, a \textit{locality structure} of $\calF$ on $\calH_{C\times I}$ is a compatible system of identifications
\[
\iota_{D, D'}: \mathcal{F}_D\boxtimes \mathcal{F}_{D'}\cong \mathcal{F}_{D\sqcup D'}, \,\ \text{for $D, D'\in \calH_{C\times I}$ with $D\cap D'=\emptyset$.} 
\]

An example of local spaces is the zastava space $\mathcal{Z}\to \mathcal{H}_{C\times I}$, recollected in detail in \cite[Section 3]{Mirk} and \cite[Section 4.2.3]{Mir2}.  Here $C$ is a smooth curve.
In \cite{Mir3}, a  modular description of $\mathcal{Z}$ is given along the lines of Drinfeld's compactification. 
Let $G$ be the semisimple simply-connected group associated to $Q$, with the choice of opposite Borel subgroups $B = TN$ and $B^- = TN^-$ with the
joint Cartan subgroup $T$. Consider the Drinfeld's compactification $\mathcal{Y}_G$ of a point:
\[
\mathcal{Y}_G = G\backslash [(G/N^+)^{\aff} \times (G/N^-)^{\aff}] / T.
\]
The zastava space $\mathcal{Z} \to \mathcal{H}_{C\times I}$ for $G$ is defined as the moduli of generic maps from $C$ to $\mathcal{Y}_G$.
Gluing the zastava spaces, one get a loop Grassmannian $\mathcal{G}r$ as a local space over $\mathcal{H}_{C\times I}$, which is a refined version of the  Beilinson-Drinfeld Grassmannian, see \cite[Section 3]{Mirk}, \cite[Section 4.2.3]{Mir2}. 

Fix a point $c\in C$, and a dimension vector $\vec{v}\in\bbN^I$,  let $[c]\in C^{(\vec{v})}\subseteq \calH_{C\times I}$ be the special divisor supported on $\{c\}$. 
The fiber $\mathcal{Z}_{[c]}$ is the Mirkovi\'c-Vilonen scheme associated to the root vector $\vec{v}$, i.e., the intersection of closures of certain semi-infinite orbits in $G_{\mathbb{C}(\!(z)\!)}/G_{\mathbb{C}[\![z]\!]}$ \cite[Section 3]{Mirk}.

The maximal torus $T\subset G$ acts on $\mathcal{Z}$. There is a certain component in a torus-fixed loci $(\mathcal{Z})^T$, which gives a section 
$\mathcal{H}_{C\times I}\subset (\mathcal{Z})^T$. We denote this component by $(\mathcal{Z})^{T\circ}$. 
The tautological line bundle $\mathcal{O}_{\mathcal{G}r}(1) |_{\mathcal{Z}^{T\circ}}$ has a natural locality structure, and 
is described in \cite[Theorem 3.1]{Mirk}. 

We now take the curve $C$ to be the elliptic curve $E$. Let $\mathcal{G}r\to \calH_{E\times I}$ be the global loop Grassmannian over $\calH_{E\times I}$.
The following theorem relies on the description of the local line bundle on $\calH_{C\times I}$ in \cite{Mirk}, \cite[Section 4.2.1]{Mir2}. 
\begin{thm}[\cite{YZell} Yang-Zhao]
\label{thmC}
\leavevmode
\begin{enumerate}
\item The classical limit $\calP^{\sph}(\Ell, Q) |_{\hbar=0}$ is isomorphic to $\mathcal{O}_{\mathcal{G}r}(1) |_{\mathcal{Z}^{T\circ}}$ as sheaves on $\mathcal{H}_{E\times I}$. 
\item The Hall multiplication $\star$ on $\calP^{\sph}(\Ell, Q)|_{\hbar=0}$ is equivalent to the locality structure on $\mathcal{O}_{\mathcal{G}r}(1) |_{\mathcal{Z}^{T\circ}}$. 
\end{enumerate}
\end{thm}

\begin{remark}
\leavevmode
\begin{enumerate}
\item Theorem \ref{thmC} is true when the curve $E$ is replaced by $\mathbb{C}$ (and $\mathbb{C^*}$), while the corresponding cohomological Hall algebra is modified to
$\calP^{\sph}(\CH, Q) |_{\hbar=0}$ (and $\calP^{\sph}(K, Q) |_{\hbar=0}$ respectively). The sheaf $\calP^{\sph}(\Ell, Q)$ deforms the local line bundle $\mathcal{O}_{\mathcal{G}r}(1) |_{\mathcal{Z}^{T\circ}}$. In the classification of local line bundles in \cite{Mirk},  $\mathcal{O}_{\mathcal{G}r}(1) |_{\mathcal{Z}^{T\circ}}$ is characterized by certain diagonal divisor of $\calH_{E\times I}$ \cite[Section 4.2.1]{Mir2}. 
As a consequence of Theorem \ref{thmC}, the shuffle formula of $\calP^{\sph}(\Ell, Q)$ gives the $\hbar$-shifting of the diagonal divisor of $\calH_{E\times I}$ that appears in $\mathcal{O}_{\mathcal{G}r}(1) |_{\mathcal{Z}^{T\circ}}$. 
\item When the base is $\mathcal{H}_{\mathbb{C}\times I}$, and cohomology theory is the Borel-Moore homology $H_{\BM}$, Theorem \ref{thmC} (1) has a similar flavour as \cite[Theorem 3.1]{BFN}. Here, we only consider $\calP^{\sph}(H_{\BM}, Q) |_{\hbar=0}$ and $\mathcal{O}_{\mathcal{G}r}(1) |_{\mathcal{Z}^{T\circ}}$ as sheaves of abelian groups. 
By Theorem \ref{thm:YZ1}, $(\calP^{\sph}(H_{\BM}, Q), \star)$ is isomorphic to the positive part of the Yangian, which is in turn related to 
$\mathcal{O}_{\mathcal{G}r}(1) |_{\mathcal{Z}^{T\circ}\to \mathcal{H}_{\mathbb{C}\times I}}$ by Theorem \ref{thmC}. 
\end{enumerate}
\end{remark}

\subsection*{Acknowledgement}
Y.Y. would like to thank the organizers of the MATRIX program \textit{Geometric R-Matrices: from Geometry to Probability} for their kind invitation,
and many participants of the program for useful discussions, including Vassily Gorbounov, Andrei Okounkov, Allen Knutson, Hitoshi Konno, Paul Zinn-Justin. Proposition \ref{prop:shuffle} and Section \ref{sec:shuffle elliptic} are new, for which we thank Hitoshi Konno for interesting discussions and communications. These notes were written when both authors were visiting the Perimeter Institute for Theoretical Physics (PI). We are grateful  to PI for the hospitality. 

\newcommand{\arxiv}[1]
{\texttt{\href{http://arxiv.org/abs/#1}{arXiv:#1}}}
\newcommand{\doi}[1]
{\texttt{\href{http://dx.doi.org/#1}{doi:#1}}}
\renewcommand{\MR}[1]
{\href{http://www.ams.org/mathscinet-getitem?mr=#1}{MR#1}}

\end{document}